\documentclass{ws-rv9x6-modified}

\usepackage{rotating_rv}

\def\hatt{\widehat}
\def\tilda{\widetilde}
\def\half{\hbox{$1\over2$}}
\def\eps{\varepsilon}
\def\RR{\mathord{I\kern-.3em R}}
\def\Var{{\rm Var}}
\def\E{{\rm E}}
\def\N{{\rm N}}
\def\d{{\rm d}}
\def\Pr{{\rm Pr}}
\def\Beta{{\rm Beta}}
\def\data{{\rm data}}
\def\Dir{{\rm Dir}}
\def\Be{{\rm Be}}
\def\be{{\rm be}}
\def\sd{{\rm sd}}

\def\midd{{\,|\,}}
\def\dell{\partial}
\def\rootn{\sqrt{n}}

\def\sumin{\sum_{i=1}^n}
\def\arr{\rightarrow}
\def\tr{{\rm t}}
\def\true{{\rm tr}}
\def\id{{\rm id}}
\def\Bernstein{{Bernshte\u\i n}}
\def\square{{\ \vrule height0.5em width0.5em depth-0.0em}}
\def\beq{\begin{eqnarray}}
\def\eeq{\end{eqnarray}}

\def\beqn{\begin{eqnarray*}}  
\def\eeqn{\end{eqnarray*}}

\begin{document}

\def\labelloc#1{\label{#1:the_current_chapter_suffix}}
\def\refloc#1{\ref{#1:the_current_chapter_suffix}}
\def\refeqloc#1{\refeq{#1:the_current_chapter_suffix}}
\setcounter{section}{0} \setcounter{equation}{0}
\setcounter{theorem}{0} \setcounter{lemma}{0}
\setcounter{corollary}{0} \setcounter{proposition}{0}
\setcounter{definition}{0} \setcounter{example}{0}
\setcounter{remark}{0} \setcounter{question}{0}
\setcounter{notation}{0}\setcounter{figure}{0} \setcounter{table}{0}

\chapter{NONPARAMETRIC QUANTILE INFERENCE \\ USING DIRICHLET PROCESSES}
\markboth{N.L. Hjort \ \& \ S. Petrone}{Dirichlet Quantile Processes}

\author{Nils Lid Hjort and Sonia Petrone}
\address{Department of Mathematics \\
University of Oslo, NORWAY\\
\medskip
IMQ, Bocconi University\\
Milano, ITALY\\
\medskip
Emails: nils@math.uio.no \& sonia.petrone@unibocconi.it}


\begin{abstract}
This chapter deals with nonparametric inference for quantiles from a
Bayesian perspective, using the Dirichlet process. The posterior
distribution for quantiles is characterised, enabling also explicit
formulae for posterior mean and variance. Unlike the Bayes estimator
for the distribution function, our Bayes estimator for the quantile
function is a smooth curve. A \Bernstein--von Mises type theorem is
given, exhibiting the limiting posterior distribution of the
quantile process. Links to kernel-smoothed quantile estimators are
provided. As a side product we develop an automatic nonparametric
density estimator, free of smoothing parameters, with support
exactly matching that of the data range. Nonparametric Bayes
estimators are also provided for other quantile-related quantities,
including the Lorenz curve and the Gini index, for Doksum's shift
curve and for Parzen's comparison distribution in two-sample
situations, and finally for the quantile regression function in
situations with covariates.\\

\noindent{\bf Keywords:} 
Bayesian bootstraps; 
Bayesian quantile regression; 
\Bernstein--von Mises theorem; 
Comparison distribution;
Dirichlet process; 
Doksum's shift function; 
Lorenz curve;
Nonparametric Bayes, 
Quantile inference.
\end{abstract}

\section{Introduction and summary}

Assume data $X_1,\ldots,X_n$ come from some unknown distribution $F$, 
and that interest focusses on one or more quantiles, 
say $Q(y)=F^{-1}(y)$. This chapter develops and discusses methods 
for carrying out nonparametric Bayesian inference for $Q$, 
based on a Dirichlet process prior for $F$. The methods also extend 
to various other quantile-related quantities in other contexts,
notably to various functions and plots for comparing two samples,
like Doksum's shift function (see Doksum, 1974a and Doksum and
Sievers, 1976) and Parzen's (1979, 1982) comparison distribution,
and to quantile regression. A guide-map of our chapter 
is as follows.

We start in Section~2 with setting the framework
and by characterising the prior and posterior distributions
of one or more quantiles. This makes it possible to
derive explicit formulae for the posterior mean,
variance and covariance in Section~3. A noteworthy feature
here is that the posterior mean function is a smooth curve
$\hatt Q(y)$, unlike the traditional Bayes estimator
$\tilda F_n$ for $F$, which has jumps at the data points.
Of particular interest is the non-informative limit
of the Bayes estimator $\hatt Q_0$ when the strength parameter
of the Dirichlet prior is sent to zero. It is seen
to be a \Bernstein-type smoothed quantile method.

In Section~4 we consider Bayes estimators of
the quantile density $q=Q'$ and of the probability density $f=F'$,
formed by the appropriate operations on $\hatt Q$.
A particular construction of interest is the density estimator
$\hatt f_0$, computed by inversion and differentiation
of $\hatt Q_0$. This estimator is nonparametric and automatic,
requires no smoothing parameters, and is supported
on the exact data range, say $[x_{(1)},x_{(n)}]$.
In Section~5 we discuss applications to the Lorenz
curve and the Gini index, which are frequently
used in econometric contexts. We obtain nonparametric
Bayes estimators of these quantities.
Then Section~6 provides Bayesian sister versions
of two important nonparametric plotting strategies
for comparing two populations: Doksum's shift
curve $D(x)$ and Parzen's comparison distribution $\pi(y)$.
Recipes for computing Bayesian credibility bands
are also given. In Section~7 we study large-sample properties
of our estimators, and reach \Bernstein--von Mises
type theorems for the limits of the posterior processes
$\rootn(Q-\hatt Q)$, $\rootn(D-\hatt D)$, $\rootn(\pi-\hatt\pi)$.
This can be used to form certain approximate credibility
intervals for the quantile function, for the shift function,
and for the comparison distribution.
Then in Section~8 results are generalised to
a semiparametric regression framework,
where the regression parameters are given a prior
independent of the quantile process
of the error distribution. Our chapter ends with a list
of concluding comments, some pointing to further
research problems of interest.

\section{The quantile process of a Dirichlet}

This section derives the basic distributional results
about the distribution of random quantiles for Dirichlet
priors, pre and post data. Our point of departure is
a Dirichlet process $F$ with parameter measure
$\alpha(\cdot)=aF_0(\cdot)$, written
$F\sim\Dir(aF_0)$, splitting into constant $a=\alpha(\RR)$
and probability distribution $F_0=\alpha/a$;
for definitions and basic results one may
consult Ferguson (1973, 1974). For a review
of general Bayesian nonparametrics, see Hjort (2003).

\subsection{Prior distributions of quantiles}

For the random $F$, consider its accompanying quantile process
$$ Q(y)=F^{-1}(y)=\inf\{t\colon F(t)\ge y\}. $$
For this left-continuous inverse of the right-continuous $F$ it
holds generally that $Q(y)\le x$ if and only if $y\le F(x)$, even
for cases when $F$, like here, has jumps. It follows, by the basic
Beta distribution property of marginals of Dirichlet processes, that
the distribution of $Q(y)$ can be written 
\beq 
\labelloc{eq:H0a}
H_{0,a}(x)
&=&\Pr\{Q(y)\le x\} \nonumber \\
&=&  1-\Be(y;aF_0(x),a\bar F_0(x))=\Be(1-y;a\bar F_0(x),aF_0(x)).
\eeq
Here and below we let $\Be(\cdot;b,c)$ and $\be(\cdot;b,c)$ denote
respectively the distribution function and the density
of a Beta variable with parameters $(b,c)$,
and $\bar F_0$ is the survival function $1-F_0$.
We allow Beta variables with parameters $(b,0)$
and $(0,c)$; these are with probability one equal to
respectively 1 and 0. Thus $\Be(y;b,0)=0$ and $\Be(y;0,c)=1$
for $y\in[0,1]$.

Note that $H_{0,a}(x)=J_a(F_0(x))$, where
$J_a(x)=\Be(1-y;a(1-x),ax)$ is the distribution of a random
$y$-quantile for the special case of $F_0$ being uniform on $(0,1)$,
say $Q_{\rm uni}(y)$. This means that the distribution of $Q(y)$ in
the general case is the same as the distribution of $F_0^{-1}(Q_{\rm
uni}(y))$. If $F_0$ has a density $f_0$, this also implies that the
prior density of $Q(y)$ is $h_0(x)=j_a(F_0(x))f_0(x)$, where 
\beq
\labelloc{eq:ja} 
j_a(x) 
=\frac{\dell}{\dell x}\int_0^{1-y}
 \frac{\Gamma(a)}{\Gamma(a-ax)\Gamma(ax)}
  u^{a-ax-1}(1-u)^{ax-1}\,\d u 
\eeq
is the density of $Q_{\rm uni}(y)$. 
The point is that the prior densities 
can be computed and displayed via numerical
integration and derivation; see 
Figure~1.

\subsection{Several quantiles simultaneously}

Consider now the joint distribution of two or more $Q$-values.
For $y_1<\cdots<y_k$, we have
\beqn
\Pr\{Q(y_1)\le t_1,\ldots,Q(y_k)\le t_k\}
&=&\Pr\{y_1\le F(t_1),\ldots,y_k\le F(t_k)\} \\
&=&\Pr\{V_1\ge y_1,\ldots,V_1+\cdots+V_k\ge y_k\},
\eeqn
in terms of a Dirichlet vector $(V_1,\ldots,V_k,V_{k+1})$
with parameters $(c_0,\ldots,c_k,c_{k+1})$,
where $c_j=aF_0(t_{j-1},t_j]$; here $F_0(A)$ is the
probability assigned to the set $A$ by the $F_0$
distribution, and $t_0=-\infty$, $t_{k+1}=\infty$.
This in principle determines all aspects of the simultaneous
distribution of the vector of random quantiles.

To give somewhat more qualitative insights into
the joint distribution of the random quantiles,
we start recalling an important and convenient property
of the Dirichlet process. When it is `chopped up'
into smaller pieces, conditioned to have certain
total probabilities on certain sets, the individual
daughter processes become independent
and are indeed still Dirichlet. In detail, if $F$ is Dirichlet $aF_0$,
and one conditions on the event $F(B_1)=z_1,\ldots,F(B_m)=z_m$,
where the $B_i$s form a partition and the $z_i$s sum to 1,
then this creates $m$ new and independent Dirichlet processes on
$B_1,\ldots,B_m$. Specifically, $F(.)/z_i$ is Dirichlet
on its `local sample space' $B_i$ with parameter $aF_0$, that is,
$$F(.)/z_i\sim{\rm Dir}(aF_0)={\rm Dir}(aF_0(B_i)\,F_0(.)/F_0(B_i)). $$
See Hjort (1986, 1996) for this fact about pinned down Dirichlets
and some of its consequences. Note the rescaling of the
Dirichlet parameter, as a new prior strength parameter $aF_0(B_i)$
times the rescaled distribution $F_0(.)/F_0(B_i)$ on set $B_i$.

Consider two quantiles $Q(y_1)$ and $Q(y_2)$, where $y_1<y_2$, for
the prior process. Conditional on $y_2=F(t_2)$, our $F$ splits into
two independent Dirichlet processes on $(-\infty,t_2]$ and
$(t_2,\infty)$. By the general result just described, and arguing as
with equation (\refloc{eq:H0a}), one finds for $t_1\le t_2$ that
\beqn 
\Pr\{Q(y_1)\le t_1\midd y_2=F(t_2)\}
&=&\Pr\{y_1\le y_2F^*(t_1)\} \\
&=&\Be(1-y_1/y_2;aF_0(t_1,t_2],aF_0(-\infty,t_1]),
\eeqn
where $F^*$ is Dirichlet $(aF_0)$ on $(-\infty,t_2]$.
This argument may be extended to the case of three or
more random quantiles, also suitable for simulation purposes.

\subsection{Posterior distributions of quantiles}

Conditionally on the randomly selected $F$, let $X_1,\ldots,X_n$ be
independently drawn from $F$. Since $F$ given data is an updated
Dirichlet with parameter $aF_0+nF_n$, where $F_n$ is the empirical
distribution of the data points, the posterior distribution of
$Q(y)$ may be written as in (\refloc{eq:H0a}), with $aF_0+nF_n$
replacing $aF_0$ there. Assume for simplicity that the data points
are distinct, order them $x_{(1)}<\cdots<x_{(n)}$, and write
$x_{(0)}=-\infty$ and $x_{(n+1)}=\infty$. Then 
\beq
\labelloc{eq:Hna} 
H_{n,a}(x)
&=&\Pr\{Q(y)\le x\midd\data\} \nonumber \\
&=&1-\Be(y;(aF_0+nF_n)(x),(a\bar F_0+n\bar F_n)(x)),
\eeq
in terms of $\bar F_0=1-F_0$ and $\bar F_n=1-F_n$.
For $x_{(i)}\le x<x_{(i+1)}$, this is equal to
$\Be(1-y;a\bar F_0(x)+n-i,aF_0(x)+i)$.
Thus $Q(y)$ has a density of the form
$$h_{n,a}(x)=(\dell/\dell x)\,\Be(1-y;a\bar F_0(x)+n-i,aF_0(x)+i)
  \quad {\rm inside\ }(x_{(i)},x_{(i+1)}), $$
cf.~the calculations leading to (\refloc{eq:ja}), 
and posterior point mass 
\beq 
\labelloc{eq:deltaHna} 
\Delta H_{n,a}(x_{(i)})
&=&\Be(y;aF_0(x_{(i)}-)+i-1,a\bar F_0(x_{(i)}-)+n-i+1) \nonumber \\
& & \qquad\qquad
   -\,\Be(y;aF_0(x_{(i)})+i,a\bar F_0(x_{(i)})+n-i) \nonumber \\
&=&(n+a)^{-1}\be(y;aF_0(x_{(i)})+i,a\bar F_0(x_{(i)})+n-i+1) 
\eeq
at point $x_{(i)}$. The partial integration formula (A1) of the Appendix
is used here, and assumes continuity of $F_0$ at $x_{(i)}$.

If $a$ is sent to zero here there is no posterior probability mass
left between data points; the distribution concentrates on the data
points with probabilities 
\beq 
\labelloc{eq:pn}
p_n(x_{(i)})&=&\Be(y;i-1,n-i+1)-\Be(y;i,n-i) \nonumber\\
&=&{n-1\choose i-1}y^{i-1}(1-y)^{n-i}.
\eeq
These binomial weights concentrate around $y$ for moderate
to large $n$.
We also have the following result, proved in our Appendix,
which says that even if $a$ is large,
the combined posterior probability that $Q(y)$ has
of landing outside the data points goes to zero
as $n$ increases. In other words, the distribution function
$H_{n,a}(x)$ becomes closer and closer to being concentrated
in only the $n$ sample points.


\begin{proposition}
\labelloc{p2.1} 
For fixed positive $a$, the sum of the
posterior point masses $\Delta H_{n,a}(x_{(i)})$ that $Q(y)$ has at
the data points goes to 1 as $n\arr\infty$. 
\end{proposition}

The prior to posterior mechanism is illustrated in 
Figure~1
for the case of the upper quartile $Q(0.75)$, 
with prior guess $F_0=\N(0,1)$, with $n=100$ data points really coming from
$\N(1,1)$. The right panel shows only the posterior probabilities
(\refloc{eq:pn}) corresponding to $a=0$; even for $a=10$ the
(\refloc{eq:deltaHna}) probabilities are quite close to those of
(\refloc{eq:pn}).

\begin{figure}
\labelloc{figure:f1}
\includegraphics[width=4.5in, height=3.0in, angle=0]{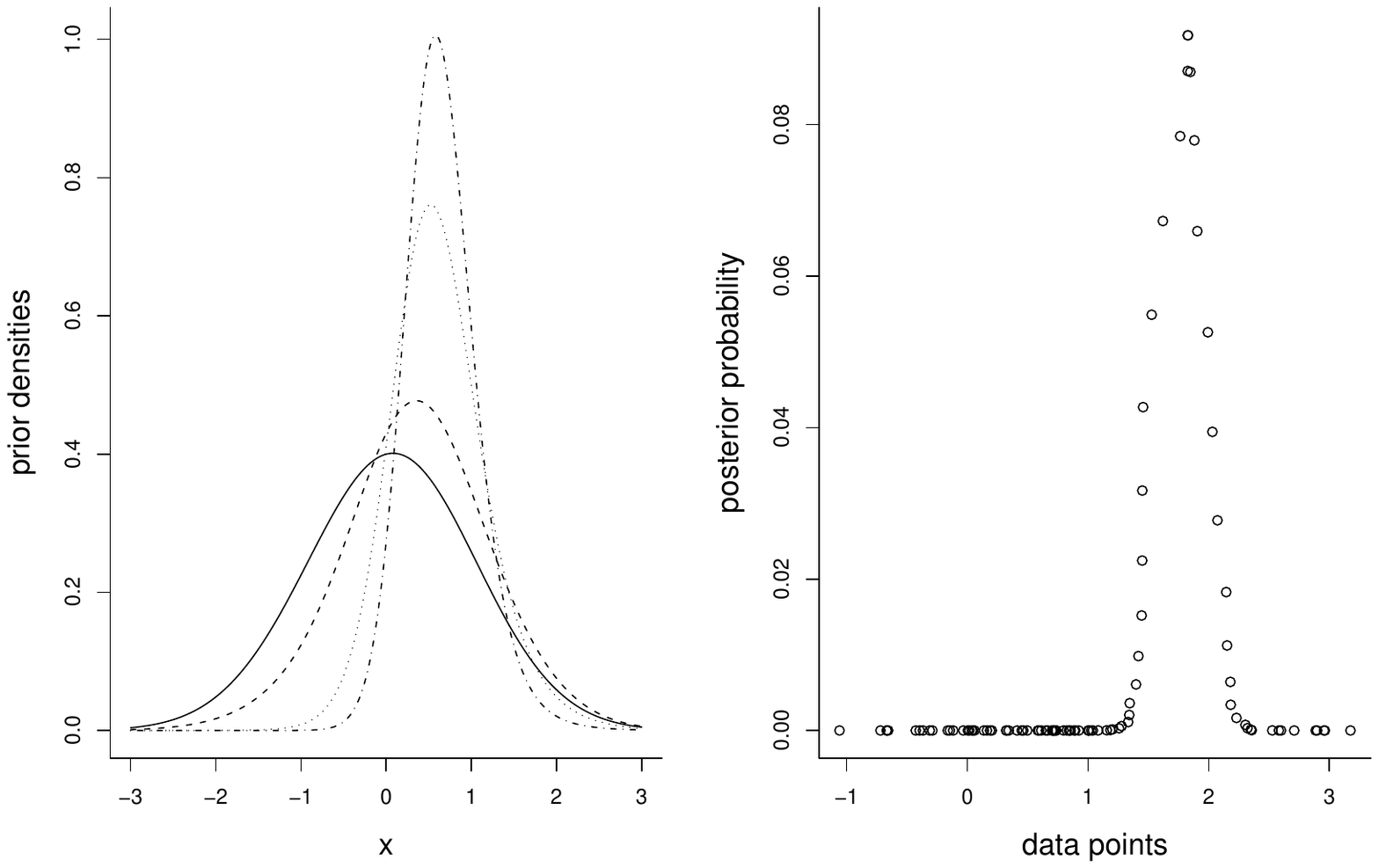}
\caption{
Prior to posterior for a given quantile: The left panel
shows the prior densi\-ties $j_a(F_0(x))f_0(x)$ at quantile
$y=0.75$, for values $a=0.1,1,5,10$, for $F_0$ the standard normal,
with smaller values of $a$ closer to the $f_0$ and larger values of
$a$ tighter around $Q_0(y)=0.675$. The right panel shows the
posterior probabilities (\refloc{eq:pn}) after having observed
$n=100$ data points from the distribution $\N(1,1)$, with true
quartile 1.675. The posterior probability mass outside the data
points equals 0.0002, 0.0017, 0.0085, 0.0181 for the four values of
$a$, respectively.}
\end{figure}

Next consider random quantiles at positions $y_1<\cdots<y_k$.
Then the event $Q(y_1)\le t_1,\ldots,Q(y_k)\le t_k$,
where $t_1\le\cdots\le t_k$, is equivalent to
$$y_1\le V_1,\,y_2\le V_1+V_2,\ldots,\,y_k\le V_1+\cdots+V_k, $$
writing now $V_j=F(t_j)-F(t_{j-1})$ for $j=1,\ldots,k+1$,
where $t_0=-\infty$ and $t_{k+1}=\infty$.
The vector $(V_1,\ldots,V_k,V_{k+1})$
has the appropriate Dirichlet distribution with parameters
$(c_1,\ldots,c_k,c_{k+1})$, where
$c_j=(aF_0+nF_n)(t_{j-1},t_j]$. This fully defines
$\Pr\{Q(y_1)\le t_1,\ldots,Q(y_k)\le t_k\midd\data\}$.
Its limit as $a\arr0$ is discussed below.

\subsection{The objective posterior quantile process}

For the non-informative prior case of $a=0$ we have seen that $Q(y)$
concentrates on the observed data points with binomial probabilities
given in (\refloc{eq:pn}). When considering two quantiles, we find
that $\Pr\{Q(y_1)=x_{(i)}\midd\data,Q(y_2)=x_{(j)}\}$ becomes
\beqn
\Be(1-y_1/y_2;j-i,i)
   &-&\Be(1-y_1/y_2;j-i+1,i-1) \\
   &=&(1/j)\be(1-y_1/y_2;j-i+1,i), 
\eeqn
using (A1) again. Combining this with (\refloc{eq:pn}) one finds
that $(Q(y_1),Q(y_2))$ selects the pair $(x_{(i)},x_{(j)})$ with
probability $p_n(x_{(i)},x_{(j)})$ equal to 
\beq
\labelloc{eq:pnbivariate} 
{(n-1)!\over (j-1)!(n-j)!}
   \!\!\!& &\!\!\!y_2^{j-1}(1-y_2)^{n-j}
   {(1/j)\,j!\over (j-i)!(i-1)!}\Bigl({y_2-y_1\over y_2}\Bigr)^{j-i}
   \Bigl({y_1\over y_2}\Bigr)^{i-1}  \nonumber \\
&=&{n-1\choose i-1,j-i,n-j}y_1^{i-1}(y_2-y_1)^{j-i}(1-y_2)^{n-j}
\eeq 
for $1\le i\le j\le n$. This trinomial structure generalises to
a suitable multinomial one for more than two quantiles at a time.

In fact, the non-informative case corresponds to a
random $F$ which is concentrated at the data points
$x_{(1)}<\cdots<x_{(n)}$ with probabilities
$D_1,\ldots,D_n$ following a Dirichlet distribution
with parameters $(1,\ldots,1)$. This in turn means that
$$Q(y)=x_{(i)} \quad
  {\rm if\ }D_1+\cdots+D_i\le y<D_1+\cdots+D_{i+1}. $$
In yet other words, $Q(y)=x_{(N(y))}$, where $N(y)$ is the smallest
$i$ at which the cumulative sum $S_i=D_1+\cdots+D_i$ exceeds $y$.
One may re-prove (\refloc{eq:pn}) from this, as well as the
trinomial result (\refloc{eq:pnbivariate}) for
$$p_n(x_{(i)},x_{(j)})
   =\Pr\{S_{i-1}<y_1\le S_i\le S_{j-1}<y_2\le S_j\}, $$
via integrations in the distribution for
$(S_{i-1},S_i-S_{i-1},S_{j-1}-S_{i-1},S_j-S_{j-1},1-S_j)$,
which is Dirichlet with parameters $(i-1,1,j-1-i,1,n-j)$.
The easiest argument uses that $S_1,\ldots,S_{n-1}$ forms
an ordered sample of size $n-1$ from the uniform distribution
on the unit interval. For the general case 
of $m$ quantiles one finds that 
$\Pr\{Q(y_1)=x_{(i_1)},\ldots,Q(y_m)=x_{(i_m)}\}$
is equal to 
\beqn
{n-1\choose i_1-1,1,\ldots,i_m-i_{m-1},1,n-i_m}
   y_1^{i_1-1}(y_2-y_1)^{i_2-i_1}\cdots(1-y_m)^{n-i_m},
\eeqn
valid for $y_1<\cdots<y_m$ and $i_1\le\cdots\le i_m$.
This `multinomial structure' hints at connections to
Brownian bridges; such are indeed studied in Section~7.

\section{Bayesian quantile inference}

To carry out Bayesian inference for $Q(y)$, for specific
quantiles or for the full quantile function, several options
are available.


One possibility is to repeatedly simulate full $Q$ functions by
numerically inverting simulated paths of $F$, these being drawn
according to the ${\rm Dir}(aF_0+nF_n)$ distribution. Another is to
work directly with the explicit posterior distribution $H_{n,a}$ of
(\refloc{eq:Hna}) for $Q(y)$, or if necessary with the
generalisations to several quantiles discussed in Section~2.3. One
attractive estimator is
$$Q_n^*(y)={\rm median}\{Q(y)\midd\data\}
  =H_{n,a}^{-1}(\half), $$
which is the Bayes estimator under loss functions
of the type $\int_0^1w(y)|\hatt Q(y)-Q(y)|\,\d y$.
It is not difficult to implement a programme that
for each $y$ finds the posterior median,
from the formula for $H_{n,a}(x)$.
For the special case of $y=\half$,
the posterior median of the random median
is the median of the posterior expectation $\tilda F_n=(aF_0+nF_n)/(a+n)$.
This may also naturally be supplemented with
posterior credibility bands of the type
$[H_{n,a}^{-1}(0.05),H_{n,a}^{-1}(0.95)]$.
It follows from theory developed below that
such a band is secured limiting 90\% pointwise
coverage probability, also in the frequentist sense.
Here, however, we focus on directly computable
Bayes estimators and on posterior variances.


We first set out to compute the posterior mean function of $Q(y)$,
which is the Bayes estimator under quadratic loss.
The informative case $a>0$ is more cumbersome mathematically
than the $a\arr0$ case, and is considered first.
Ferguson (1973, p.~224) pointed out that the posterior expectation 
``is difficult to compute, and may, in fact, not even exist''.
Here we give both precise finiteness conditions
and a formula; such have apparently not been
given earlier in the literature.
From our results in Section~2 it is clear that
when the integrals exist, a formula for the posterior mean
takes the form
\begin{equation}
\labelloc{eq:Qahat} \hatt Q_a(y)=\sumin \Delta
H_{n,a}(x_{(i)})x_{(i)}
  +\sum_{i=0}^n\int_{(x_{(i)},x_{(i+1)})} xh_{n,a}(x)\,\d x,
\end{equation}
with $H_{n,a}$ and $h_{n,a}$ as given in Section~2.3.
Existence requires finiteness
of the first and the last integrals here,
over respectively $(-\infty,x_{(1)})$ and $(x_{(n)},\infty)$.
The following is proved in our Appendix.


\begin{proposition}
\labelloc{p3.1} 
Let $Q=F^{-1}$ have the prior process
induced by a Dirichlet process prior with parameter $aF_0$ for $F$,
where $a$ is positive. Then the posterior mean $\hatt Q_a(y)$ of the
quantile function $Q(y)$ is well-defined and finite if and only if
the prior mean $\E_0|X|=\int |x|\,\d F_0(x)$ is finite. This result
is independent of the sample size $n$ and of the value of $y$, and
is also valid for the prior situation.
\end{proposition}

For implementation purposes, formula (\refloc{eq:Qahat}) is a little
awkward. A simpler equivalent formula is 
\beq
\labelloc{eq:Qhatformula} 
\hatt Q_a(y) 
&=&\int_0^\infty\Pr\{Q(y)\ge
x\midd\data\}\,\d x
   -\int_{-\infty}^0\Pr\{Q(y)\le x\mid\data\}\,\d x  \nonumber \\
&=&\int_0^\infty\Be(y;aF_0(x)+nF_n(x),a\bar F_0(x)+n\bar F_n(x))\,\d x \\
& &\qquad\qquad
   -\int_{-\infty}^0\Be(1-y;a\bar F_0(x)+n\bar F_n(x),
   aF_0(x)+nF_n(x))\,\d x. \nonumber
\eeq 
For large $a$ dominating $n$ in size, this estimator is close
to the prior guess function $F_0^{-1}(y)$. Even a moderate or large
$a$ will however be `washed out' by the data as $n$ grows, as is
apparent from Proposition \refloc{p2.1} and made clearer in
Section~7.


Particularly interesting is the nonparametric quantile estimator
emerging by letting $a$ tend to zero, since the posterior then
concentrates on the data points alone. By (\refloc{eq:pn}), the
result is 
\beq 
\labelloc{eq:Q0hat} 
\hatt Q_0(y)=\sumin {n-1\choose i-1}
y^{i-1}(1-y)^{n-i}x_{(i)}. 
\eeq 
This is a $(n-1)$-degree
polynomial function that smoothly climbs from $\hatt Q_0(0)=x_{(1)}$
to $\hatt Q_0(1)=x_{(n)}$. It may of course be used also outside the
present Bayesian framework. Its frequentist properties have been
studied, to various extents, in Hjort (1986), Sheather and Marron
(1990), and Cheng (1995), and we learn more in Section~7 below.
Interestingly, it can also be expressed as
$n^{-1}\sumin\be(y;i,n-i+1)\,x_{(i)}$, 
an even mixture of beta densities.


The posterior variance $\hatt V_a(y)$ may also be computed explicitly, via
$\E\{Q(y)^2\midd\data\}
   =\int_0^\infty\Pr\{|Q(y)|\ge x^{1/2}\midd\data\}\,\d x$, 
which as with other calculations above with some efforts also may be
expressed in terms of finite sums of explicit terms. One may show as
with Proposition \refloc{p3.1} that the posterior variance is finite
if and only if the prior variance is finite; this statement is valid
for each $n$. In the $a\arr0$ case the variance simplifies to 
\beq
\labelloc{eq:V0} 
\hatt V_0(y)=\sumin {n-1\choose i-1}y^{i-1}(1-y)^{n-i}\,
  \{x_{(i)}-\hatt Q_0(y)\}^2.
\eeq 
The posterior covariance between two quantiles can similarly be
estimated explicitly, via (\refloc{eq:pnbivariate}). With the
limiting normality results of Section~7 this implies for example
that $\hatt Q_0(y)\pm 1.96\,\hatt V_0(y)^{1/2}$ becomes an
asymptotic pointwise 95\% confidence band in the frequentist sense,
as well as an asymptotic pointwise 95\% credibility band in the
Bayesian posterior sense.


\begin{remark}
\labelloc{r3.2}
Note first that $X_{([nt])}$ is distributed as
$F_\true^{-1}(U_{([nt])})$, in terms of an ordered sample
$U_{(1)},\ldots,U_{(n)}$ from the uniform distribution
on the unit interval, in terms of the true distribution
$F_\true$ for the $X_i$s. Hence $X_{([nt])}$
is close to $F^{-1}(t)$ for moderate to large $n$.
A kernel type estimator based on the
order statistics would be of the form
$$\tilda Q(y)=\int K_h(t-y)X_{([nt])}\,\d t
        \doteq n^{-1}\sumin K_h(i/n-y)x_{(i)}, $$
in terms of a scaled kernel function $K_h(u)=h^{-1}K(h^{-1}u)$ and
its smoothing parameter $h$. One may now show, via approximate
normality of the binomial weights used in (\refloc{eq:Q0hat}), that
$\hatt Q_0(y)$ is asymptotically identical to such a kernel
estimator, with $K$ the standard normal kernel, and
$h=\{y(1-y)/n\}^{1/2}$; proving this is related to the classic de
Moivre--Laplace result. This means under-smoothing if compared to
the theoretically optimal bandwidths, which are of size
$O(n^{-1/3})$ for moderate to large $n$. 
See Sheather and Marron (1990).~\square 
\end{remark}

\section{Quantile density and probability density estimators}

Assume that the true $F=F_\true$ governing data has
a smooth density $f_\true$, positive on its support.
The quantile function $Q_\true(y)=F_\true^{-1}(y)$ has derivative
$q_\true(y)=1/f_\true(Q_\true(y))$,
sometimes called the quantile density function.
In this section we look at the relatives
$\hatt q_a$ and $\hatt f_a$ following from $\hatt Q_a$
of the previous section, with $a=0$ leading to
particularly interesting estimators.


First consider the quantile density.  
The Bayes estimator with the Dirichlet process prior
under squared error loss is, via results of Section~3,
after an exchange of derivative and mean operations,
\beqn
\hatt q_a(y)
&=&\int_0^\infty\be(y;aF_0(x)+nF_n(x),
   a\bar F_0(x)+n\bar F_n(x))\,\d x \\
& &\qquad
   +\int_{-\infty}^0\be(1-y;a\bar F_0(x)+n\bar F_n(x),
   aF_0(x)+nF_n(x))\,\d x.
\eeqn 
The limiting non-informative case $\hatt q_0=\hatt Q_0'$ can
be written in several revealing ways, from (\refloc{eq:Q0hat}) or as
a limit of the above; 
\beqn 
\hatt q_0(y) &=&\sumin {n-1\choose i-1}y^{i-1}(1-y)^{n-i}
        \Bigl({i-1\over y}-{n-i\over 1-y}\Bigr)\,x_{(i)} \\
&=&\int_{x_{(1)}}^{x_{(n)}}\be(y,nF_n(x),n\bar F_n(x))\,\d x
 =\sum_{i=1}^{n-1}(x_{(i+1)}-x_{(i)})\be(y,i,n-i).
\eeqn 
Note that there is no smoothing parameter in this
construction; the inherent smoothing comes `for free' through the
limiting Dirichlet process prior argument. The level of this
inherent smoothing is about $\{y(1-y)/n\}^{1/2}$, 
as per Remark \refloc{r3.2} above.


We have devised Bayesian ways of estimating $Q=F^{-1}$, 
and are free to invert back to the $F$ scale, 
finding in effect new estimators of the distribution function. 
Thus let $\hatt F_a(x)$ be the
solution to $x=\hatt Q_a(y)$. It can be computed from
(\refloc{eq:Qhatformula}). This is not the same as the posterior
mean or posterior median, but is a Bayes estimator in its own right,
with loss function of the form $L(F,\hatt F)=\int_0^1w(\hatt
Q-Q)^2\,\d y$. It is noteworthy that $\hatt F_a$ is smooth and
differentiable in $x$, unlike the posterior mean function
$\{aF_0(x)+nF_n(x)\}/(a+n)$, which has jumps at each data point.
When $a$ dominates $n$, $\hatt F_a$ is close to $F_0$. The case
$a=0$ is again of particular interest, with $\hatt F_0$ climbing
smoothly from zero at $x_{(1)}$ to one at $x_{(n)}$, with an
everywhere positive density over this data range. The $\hatt F_0$
may be considered a smoother default alternative to the empirical
distribution function $F_n$, for e.g.~display purposes. It follows
from theory of Section~7 that $\rootn(\hatt F_0-F_n)\arr_p0$, so the
two estimators are close.

It is well known that distribution functions chosen from the
Dirichlet prior are discrete with probability one.
Thus the random posterior quantile process is also discrete.
That the posterior mean of $Q(y)$ happens to be a smooth function
of $y$ is not a contradiction, however.
We have somehow `gained smoothness' by passing from $F$ to $Q$ and back
to $F$ again. This should perhaps be viewed as mathematical
happenstance; neither $F$ nor $Q$ is smooth,
but the posterior mean function of $Q$ is.


Our efforts also lead to new nonparametric Bayesian
density estimators. 
We solved $\hatt Q_a(y)=x$ to reach the estimator $\hatt F_a(x)$,
and its derivative $\hatt f_a(x)$ is a Bayes estimator
of the underlying data density $f_\true$.
The result is a continuous bridge in $a$,
from the prior guess $f_0$ for $a$ large to
something genuinely nonparametric and prior-independent
for $a=0$. One may contemplate devising methods
for choosing $a$ from data, smoothing between
prior and data, perhaps in empirical Bayesian
fashions, or via a hyperprior. Here we focus on
the automatic density estimator $\hatt f_0$,
corresponding to the non-informative prior.

From $\hatt f_0(x)=(\hatt Q_0^{-1})'(x)$ we may write 
\beq
\labelloc{eq:f0hat} 
\hatt f_0(x)=\Bigl[\sum_{i=1}^{n-1}(x_{(i+1)}-x_{(i)})
   \be(\hatt F_0(x);i,n-i)\Bigr]^{-1}, 
\eeq
where, for each $x$, the equation $\hatt Q_0(y)=x$
is numerically solved for $y$ to get $\hatt F_0(x)$,
for example using a Newton--Raphson method.
From smoothness properties of $\hatt F_0$ noted above,
one sees that $\hatt f_0(x)$ is strictly positive
on the exact data range $[x_{(1)},x_{(n)}]$,
with unit integral.

The formula above for $\hatt f_0(x)$ is directly valid
inside $(x_{(1)},x_{(n)})$.
At the end points some details reveal that
\beqn
\hatt f_0(x_{(1)})
   &=&1/\hatt q_0(0)=\{(n-1)(x_{(2)}-x_{(1)})\}^{-1}, \\
\hatt f_0(x_{(n)})
   &=&1/\hatt q_0(1)=\{(n-1)(x_{(n)}-x_{(n-1)})\}^{-1}. 
\eeqn
It is interesting and perhaps surprising
that this nonparametric Bayesian approach
leads to such explicit advice about the behaviour of $f$
near and at the endpoints;
estimation of densities in the tails is in general
a difficult problem with no clear favourite among
frequentist proposals.

It is perhaps too adventurous to struggle for the abolition of all
histograms, replacing them instead with the automatic Bayesian
non-informative density estimator $\hatt f_0$ of
(\refloc{eq:f0hat}). But as 
Figure~2 
illustrates, it can be a successful data descriptor, 
with better smoothness properties
than the histogram, and without the need for selecting smoothing
parameters. It also has the pleasant property that $\int x\hatt
f_0(x)\,\d x$ is precisely equal to the data mean $\bar x$. When
compared to traditional kernel methods it will be seen to smooth
less, actually with an amount corresponding to a locally varying
bandwidth of size $O(n^{-1/2})$, as opposed to the traditional
optimal size $O(n^{-1/5})$. The latter does assume two derivatives
of the underlying density, however, whereas the (\refloc{eq:f0hat})
estimator has been constructed directly from the data, without any
further smoothness assumptions.

\begin{figure} 
\labelloc{figure:f2}\center
\labelloc{fig:density}
\includegraphics[width=4.5in, height=3.0in, angle=0]{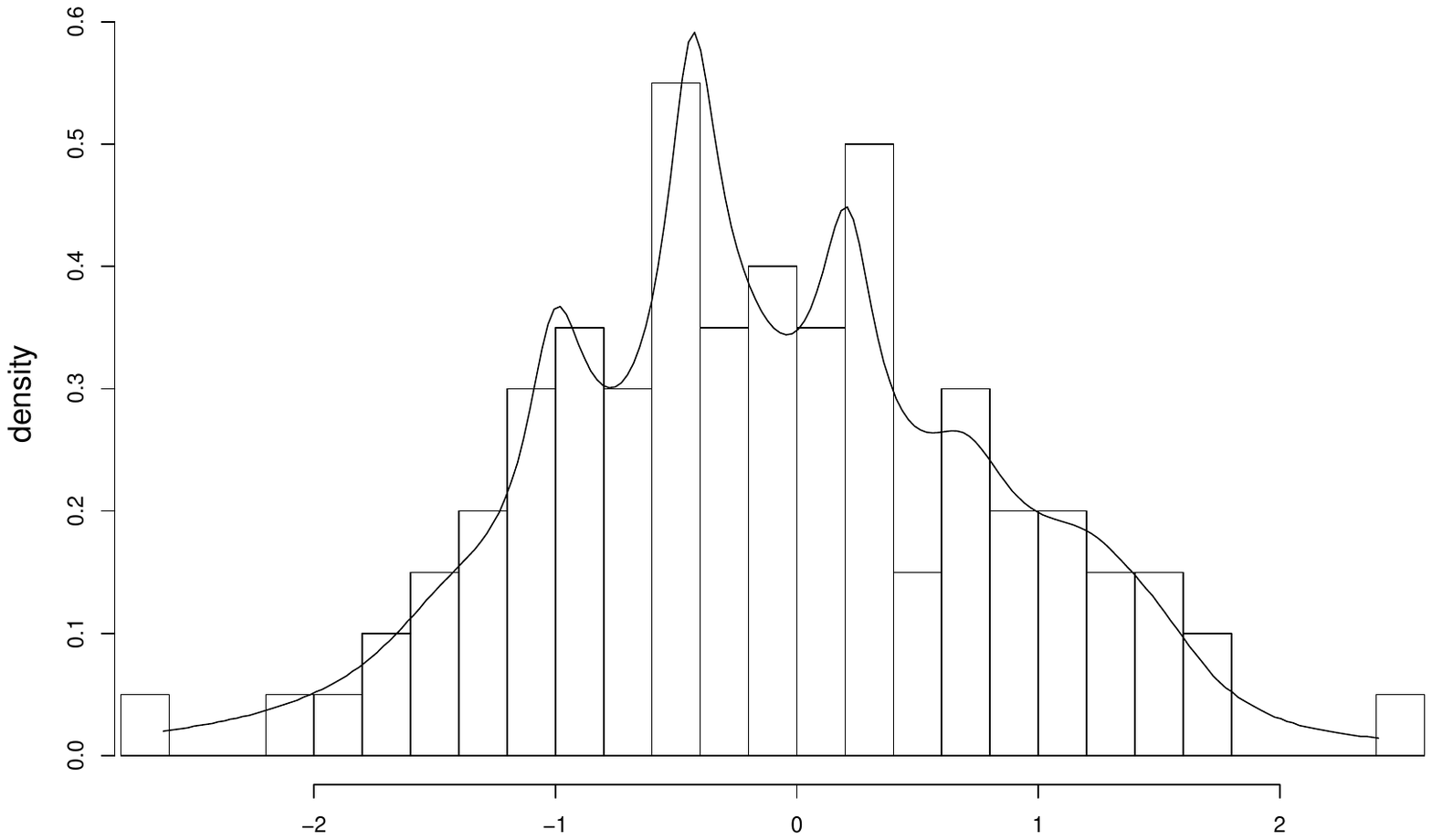}
\caption{
A histogram (with more cells than usual) over $n=100$ data
points from the standard normal, along with the automatic density
estimator of (\refloc{eq:f0hat}).}
\end{figure}

\def\ML{{\rm ml}}

\section{The Lorenz curve and the Gini index}

Quantile functions are used in many spheres of
theoretical and applied statistics. One such
is that of econometric studies of income
distributions, where information is often quantified
and compared in terms of the so-called Lorenz curve
(going back a hundred years, to Lorenz, 1905), along with
various summary measures, like the Gini index;
see e.g.~Aaberge (2001) and Aaberge, Bjerve and Doksum (2005).
This section considers nonparametric Bayes inference
for such curves and indices.


When the distribution $F$ of data is supported on
the positive halfline, the {\it Lorenz curve} 
is defined as
$$ L(y)=\int_0^y Q(u)\,\d u\Big/\int_0^1 Q(u)\,\d u
   \quad {\rm for\ }0\le y\le 1. $$
The numerator is also equal to $\int_0^{Q(y)}x\,\d F(x)$,
and the denominator is simply equal to the mean $\mu$ of the distribution.
It is in general convex, and is equal to the diagonal $L(y)=y$
if and only if the underlying distribution is concentrated
in a single point (perfect equality of income).

Bayesian inference can now be carried out for $L$,
for example through simulation of $Q$ curves from
the posterior distribution. A natural Bayes estimator
takes the form
$$ \hatt L_a(y)=\int_0^y\hatt Q_a(u)\,\d u\Big
   /\int_0^1\hatt Q_a(u)\,\d u, $$
stemming from keeping the weighted squared error loss function
for $Q$, transforming the solution to $L$ scale.
Particularly interesting is the non-informative limit version
$$\hatt L_0(y)={\int_0^y\hatt Q_0(u)\,\d u\over \int_0^1\hatt Q_0(u)\,\d u}
   =\Bigl\{n^{-1}\sumin \Be(y;i,n-i+1)x_{(i)}\Bigr\}\Big/\bar x
  \quad {\rm for\ }0\le y\le 1. $$


The {\it Gini index} is a measure of closeness of
the $L$ curve to the diagonal, i.e.~the egalitarian case,
and is defined as
$G=2\int_0^1\{y-L(y)\}\,\d y$.
With a Dirichlet prior for $F$ and any weighted
integrated squared error loss function for the
quantile function, we get a Bayes estimator
$\hatt G_a=2\int_0^1\{y-\hatt L_a(y)\}\,\d y$.
The non-informative limiting version is of particular interest.
Some algebra shows that 
$\hatt G_0=2\int_0^1\{y-\hatt L_0(y)\}\,\d y$
may be expressed as
\beqn
\hatt G_0
 =1-2{1\over n}\sumin\Bigl(1-{i\over n+1}\Bigr){x_{(i)}\over \bar x}
 =2{1\over n}\sumin{i\over n+1}{x_{(i)}\over \bar x}-1.
\eeqn
Its value may be supplemented with a credibility interval
via posterior simulation of $L$ curves.

\section{Doksum's shift and Parzen's comparison}

Assume data $X_1',\ldots,X_m'$ come from the distribution $G$,
independently of $X_1,\ldots,\allowbreak X_n$ from $F$.
When inspecting such data there are various options for portraying,
characterising and testing for differences
between the two distributions. 


Doksum (1974a) introduced the so-called {\it shift function}
$$D(x)=G^{-1}(F(x))-x. $$
Its essential property is that $X+D(X)$ has the same
distribution as $X'$. The shift function has a
particularly useful role in situations with
control and treatment groups. If the distributions
of $X$ and $X'$ differ only in location, for example,
then $D(x)$ is constant; if on the other hand $G$
is a location-and-scale translation of $F$, then $D(x)$ is linear.
Doksum (1974a) studied the natural nonparametric estimator
$\tilda D(x)=G_m^{-1}(F_n(x))-x$, in terms of the
empirical cumulative distributions $F_n$ and $G_m$;
see Section~7.3 below for its key large-sample properties.
Here we describe how Bayesian inference can be
carried out, starting with independent
priors $F\sim\Dir(aF_0)$ and $G\sim\Dir(bG_0)$.

The posterior distribution at a fixed $x$ is
$$K_{m,n}(t)=\Pr\{G^{-1}(F(x))-x\le t\midd\data\}
      =\Pr\{F(x)\le G(x+t)\midd\data\}, $$
which can be evaluated via numerical integration,
using the Beta distributions involved. For the
non-informative case,
\beqn
K_{m,n}(t)
&=&\Pr\{\Beta(nF_n(x),n\bar F_n(x))
   \le \Beta(mG_m(x+t),m\bar G_m(x+t))\} \\
&=&\int_0^1\Be(g,nF_n(x),n\bar F_n(x))
   \be(g,mG_m(x+t),m\bar G_m(x+t))\,\d g.
\eeqn
This can be used to compute the posterior median
estimator $K_{m,n}^{-1}(\half)$, along with a pointwise
credibility band, say $[K_{m,n}^{-1}(0.05),K_{m,n}^{-1}(0.95)]$.
It follows from results of Section~7 that such
a band will have frequentist coverage level
converging to the required 90\%, for each $x$,
when the sample sizes grow.

We also provide formulae for the posterior
mean and variance, for the non-informative case.
These are found by first conditioning on $F$, viz.
\beqn
\E\{G^{-1}(F(x))\midd\data,F\}
&=&\sum_{j=1}^m{m-1\choose j-1}F(x)^{j-1}\bar F(x)^{m-j}x_{(j)}', \\
\E\{G^{-1}(F(x))^2\midd\data,F\}
&=&\sum_{j=1}^m{m-1\choose j-1}F(x)^{j-1}\bar F(x)^{m-j}(x_{(j)}')^2.
\eeqn
Using Beta moment formulae this gives 
the Bayes estimator $\hatt D_0(x)$ as
\beqn
\sum_{j=1}^m{m-1\choose j-1}
   {\Gamma(n)\over \Gamma(nF_n)\Gamma(n\bar F_n)}
   {\Gamma(nF_n+j-1)\Gamma(n\bar F_n+m-j)
   \over \Gamma(n+m-1)}x_{(j)}'-x, 
\eeqn 
writing $F_n$ and $\bar F_n$ for $F_n(x)$ and $\bar F_n(x)$, 
while the posterior variance $\hatt V_0(x)$ can be found as
\beqn 
\sum_{j=1}^m{m-1\choose j-1}
   \!\!& &\!\!{\Gamma(n)\over \Gamma(nF_n)\Gamma(n\bar F_n)}
   {\Gamma(nF_n+j-1)\Gamma(n\bar F_n+m-j)
   \over \Gamma(n+m-1)}(x_{(j)}')^2 \\
& &\qquad\qquad\qquad\qquad -\,\{\hatt D_0(x)+x\}^2.
\eeqn
The theory of Section~7 guarantees that the
band $\hatt D_0(x)\pm1.645\,\hatt V_0(x)^{1/2}$
has pointwise coverage level converging to 90\%,
for example, as the sample sizes increase.

\begin{figure}
\labelloc{figure:f3}
\includegraphics[width=4.5in, height=3.0in, angle=0]{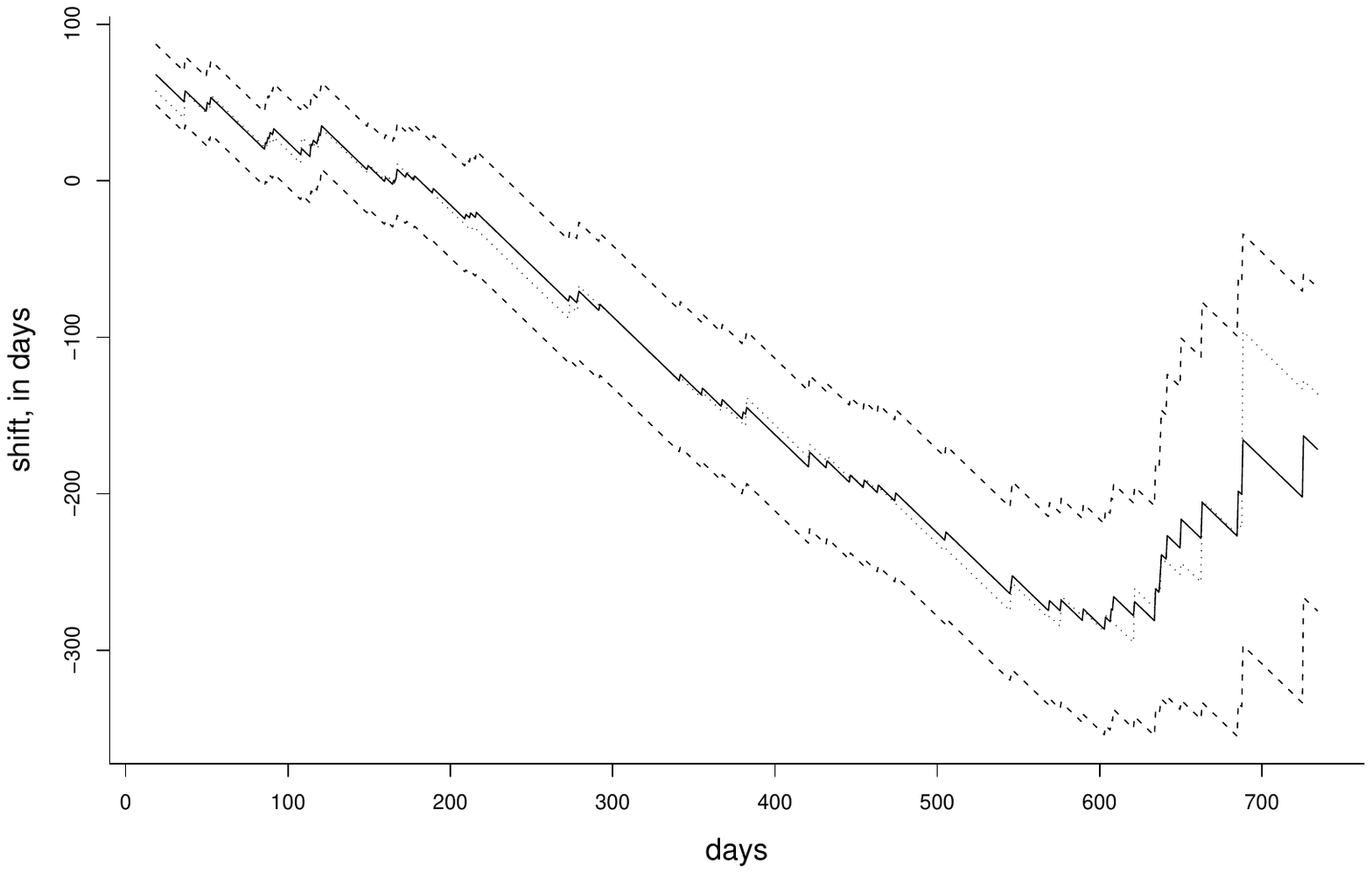}
\caption{
For the 65 guinea pigs in the control group
and the 60 in the treatment group, we display
the Bayes estimator [full line]
of the shift function associated with the
two survival distributions, alongside Doksum's
sample estimator [dotted line].
Also given is the approximate pointwise 90\%
credibility band.}
\end{figure}

Doksum (1974a) illustrated his shift function using survival data of
guinea pigs in Bjerkedal's (1960) study of the effect of virulent
tubercle bacilli, with 65 in the control group and 60 in the
treatment group, the latter receiving a dose of such bacilli. 
Here we re-analyse Bjerkedal and Doksum's data, with 
Figure~3
displaying the Bayes estimate $\hatt D_0(x)$, seen there to be quite
close to Doksum's direct estimate. Also displayed is the approximate
90\% pointwise confidence band. The figure illustrates dramatically
that the weaker pigs (those who tend to die early) will tend to have
longer lives with the treatment, while the stronger pigs (those
whose lives tend to be long) are made drastically weaker, i.e.~their
life lengths will decrease. This analysis agrees with conclusions in
Doksum (1974a). For example, pigs with life expectancy around 500
days can expect to live around 200 days less if receiving the
virulent tubercle bacilli in question.


Parzen (1979, 1982, 2002) has repeatedly advocated analysing and
estimating the function $\pi(y)=G(F^{-1}(y))$, which he terms 
the {\it comparison distribution}. 
This function, or estimates thereof, may be plotted 
against the identity function $\pi_\id(y)=y$ on the unit
interval; equality of the two distributions is equivalent to
$\pi=\pi_\id$. See also Newton's interview with Parzen (2002,
p.~372--374). We now consider nonparametric Bayesian estimation of
the Parzen curve via independent Dirichlet process priors on $F$ and
$G$, with parameters respectively $aF_0$ and $bG_0$.

A formula for the posterior mean $\hatt\pi(y)$
may be derived as follows. Let $\hatt G_m=w_m'G_0+(1-w_m')G_m$
be the posterior mean of $G$, in terms of $w_m'=b/(b+m)$
and the empirical distribution $G_m$ for the $m$ data points.
Then $\hatt\pi(y)$ is the mean of $\E\{G(Q(y))\midd Q,\data\}$,
i.e.~the mean of $\hatt G_m(Q(y))$ given data, leading to
\beqn
\hatt\pi(y)
&=&w_m'\E\{G_0(Q(y))\midd\data\}+(1-w_m')\E\{G_m(Q(y))\midd\data\} \\
&=&w_m'\int_0^1\Pr\{G_0(Q(y))>z\midd\data\}\,\d z \\
& &  \quad +\ (1-w_m'){1\over m}\sum_{j=1}^m\Pr\{x_j'\le Q(y)\midd\data\} \\
&=&w_m'\int_0^1\Be(y;(aF_0+nF_n)(G_0^{-1}(z)),
   (a\bar F_0+n\bar F_n)(G_0^{-1}(z)))\,\d z \\
& & \quad +\,(1-w_m'){1\over m}\sum_{j=1}^m\Be(y;(aF_0+nF_n)(x_j'-),
   (a\bar F_0+n\bar F_n)(x_j'-)),
\eeqn
where the second term is explicit and the first
not difficult to compute numerically.
If there are no ties between the $x_j'$ and the $x_i$ points
for the two samples, $(aF_0+nF_n)(x_j'-)$ is the same
as $(aF_0+nF_n)(x_j')$.
For the non-informative case of $a$ and $b$ both going to zero,
we have the particularly appealing estimator
$$\hatt\pi_0(y)=
{1\over m}\sum_{j=1}^m\Be(y;nF_n(x_j'-),n\bar F_n(x_j'-)). $$
Its derivative, which is an estimate of what Parzen terms
the comparison density $g(F^{-1}(y))\allowbreak/f(F^{-1}(y))$,
provided the densities $g=G'$ and $f=F'$ exist,
is quite simply
$(1/m)\sum_{j=1}^m\be(y;nF_n(x_j'-),n\bar F_n(x_j'-))$. 
The posterior variance of $\pi(y)$ may also be calculated
with some further efforts. For the non-informative case of
$a=b=0$, we find
\beqn
\lefteqn{\Var\{\pi(y)\midd\data\}
={1\over m+1}\hatt\pi_0(y)\{1-\hatt\pi_0(y)\}} \\
& & \qquad\qquad\quad 
   +\,{m\over m+1}\Bigl\{{1\over m^2}\sum_{j,k}\Be(y;nF_n(x_{j,k}'-),
   n\bar F_n(x_{j,k}'-))-\hatt\pi_0(y)^2\Bigr\},
\eeqn
in which $x_{j,k}'=\max(x_j',x_k')$.

It is seen that $\hatt\pi_0(y)$ provides a smoother
alternative to the direct nonparametric Parzen
estimator. The theory of Section~7 implies that
the two estimators are asymptotically equivalent,
and also that the simple credibility band
$\hatt\pi_0(y)\pm1.96\,\hatt\sd(y)$, with
$\hatt\sd(y)$ the posterior standard deviation
computed as above, is a band reaching 95\% level
coverage, in both the frequentist and Bayesian
settings, as sample sizes grow.


Laake, Laake and Aaberge (1985) discussed relations
between hospitalisation, as a measure of morbidity,
and mortality. The patient material consisted of
367 consecutive admissions at hospitals in Oslo
in 1980 (176 males and 191 females), while data
on mortality in Oslo consisted of 6140 deaths
(2989 males and 3151 females). Letting $F$ be
the distribution of age at hospitalisation and $G$
the distribution of age at death, Laake, Laake and Aaberge
suggested studying $\Lambda(y)=G^{-1}(y)-F^{-1}(y)$,
a direct comparison of the two quantile functions.
It is a close cousin of the Doksum curve in that
$\Lambda(F(x))=D(x)$.

We have re-analysed the data of Laake, Laake and Aaberge 
(1985, Table~1) using the Bayes estimator 
$\hatt\Lambda(y)=\hatt Q_G(y)-\hatt Q_F(y)$, 
with components as in (\refloc{eq:Q0hat}). 
For our illustration, we `made' continuous data from their table, 
by distributing the number of observations in question evenly over the
required age interval; thus 12 and 17 observed hospitalised women in
the age groups 50--54 and 55--59 gave rise to 12 and 17 $X$s spread
uniformly on the intervals $[49.5,54.5]$ and $[54.5,59.5]$, and so on. 
Figure~4 
presents these curves, for women and for men
separately, along with confidence band
$\hatt\Lambda(y)\pm1.96\,\hatt\sd(y)$, where $\hatt\sd(y)^2$ is the
sum of the two variance estimates involved, computed as in
(\refloc{eq:V0}). It follows from the theory of Section~7 that this
band indeed has the intended approximate 95\% confidence level at
each quantile value $y$. The analysis shows that to the first order
of approximation, and apart from noticeable deviations for the very
young and the very old, age at hospitalisation and age at death are
similar, with a constant shift between them, about seven years for
women and six years for men. This interpretation is in essential
agreement with conclusions reached by Laake, Laake and Aaberge.

\begin{figure}
\labelloc{figure:f4}
\includegraphics[width=4.5in, height=3.0in, angle=0]{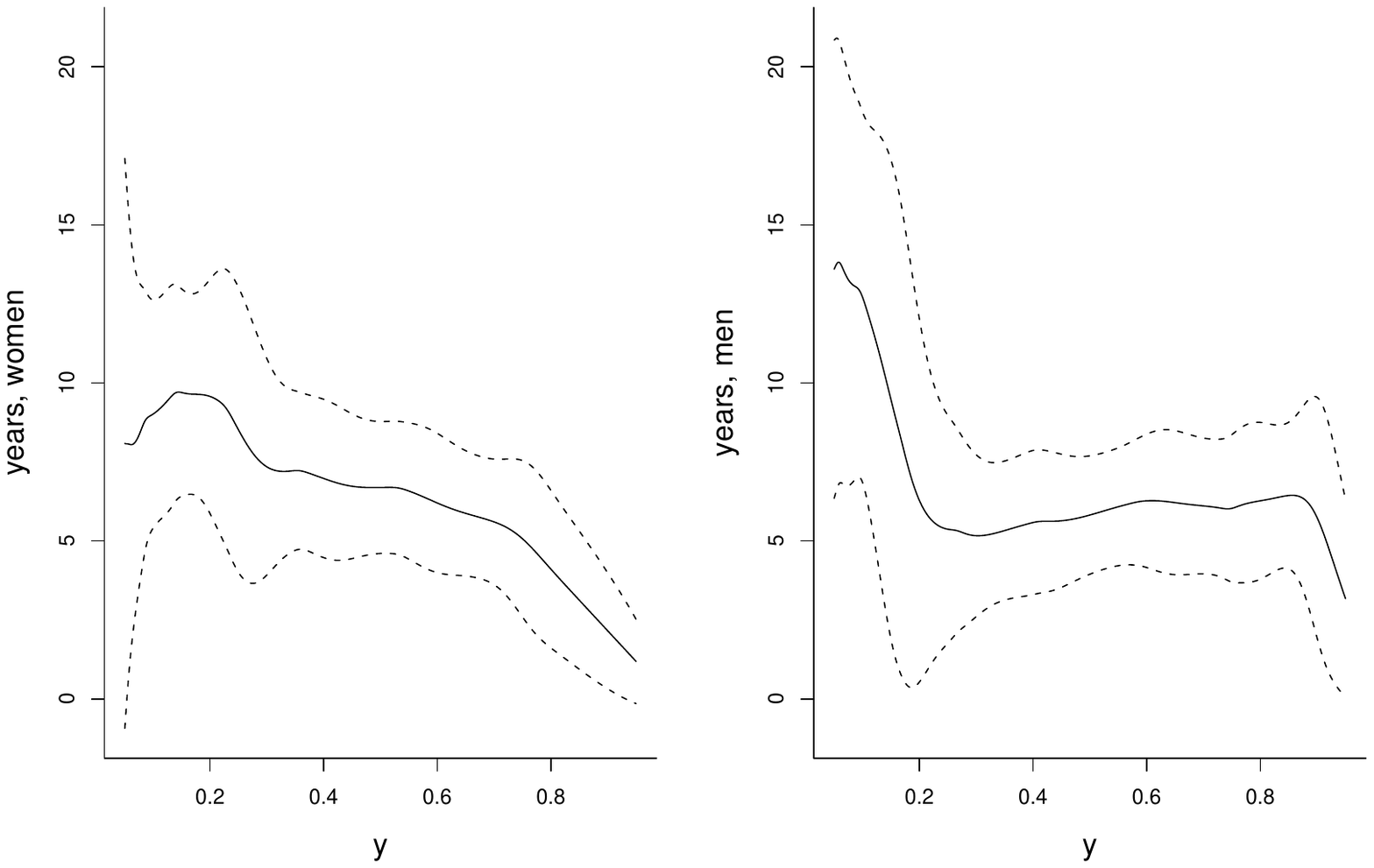}
\caption{
Estimated quantile difference $G^{-1}(y)-F^{-1}(y)$
between age at death distribution and age at hospitalisation
distribution, along with pointwise 95\% confidence bands,
for women (left) and for men (right).}
\end{figure}

\section{Large-sample analysis}

In this section we discuss large-sample behaviour of
the estimation schemes we have developed, from
both the Bayesian and frequentist perspectives.

\subsection{Nonparametric \Bernstein--von Mises theorems}

To set results reached below in perspective, it is useful
first to recall some well-known results about
the limiting behaviour of maximum likelihood and Bayes estimators,
as well as about the posterior distribution,
valid for general parametric models.
Specifically, assume i.i.d.~data $Z_1,\ldots,Z_n$ follow
a parametric density $g(z,\theta)$, with $\theta_\true$
the true parameter, and let $\hatt\theta_{\ML}$
and $\hatt\theta_B$ be the maximum likelihood and
posterior mean Bayes estimator under a suitable prior
$\pi(\d\theta)$. Then, under mild regularity conditions,
discussed e.g.~in Bickel and Doksum (2001, Ch.~5--6),
four notable results are valid:
(i) $\rootn(\hatt\theta_{\ML}-\theta_\true)\arr_d\N(0,J(\theta_\true)^{-1})$;
(ii) $\rootn(\hatt\theta_B-\hatt\theta_{\ML})\arr_p0$;
(iii) with probability one, the posterior distribution is such that
$\rootn(\theta-\hatt\theta_B)\midd\data\arr_d\N(0,J(\theta_\true)^{-1})$.
Here $J(\theta)$ is the information matrix of the model,
see e.g.~Bickel and Doksum (2001, Ch.~6).
With a consistent estimator $\hatt J$ of this matrix
one may compute the approximation $\N(\hatt\theta_{\ML},n^{-1}\hatt J)$
to the posterior distribution of $\theta$. Result (iv)
is that this simple method is first-order asymptotically correct,
i.e.~$\hatt J^{-1/2}(\theta-\hatt\theta_{\ML})\midd\data$
goes a.s.~to $\N(0,I)$, the implication being that
one may approximate the posterior distribution
without carrying out the Bayesian updating calculations as such.
Results of the (iii) and (iv) variety
are often called \Bernstein--von Mises theorems;
see e.g.~LeCam and Yang (1990, Ch.~7).
Note that Bayes and maximum likelihood estimators have
the same limit distribution, regardless also of the
prior one starts out with, as a consequence of (ii).

Such statements and results become more complicated in
non- and semiparametric models, and sometimes do not hold.
There are situation when Bayes solutions do not match
the natural frequentist estimators, and other situations where
the posterior distribution goes awry, or have a limit
different from that indicated by \Bernstein--von Mises
heuristics; see e.g.~Diaconis and Freedman (1986a, 1986b),
Hjort (1986, 1996, 2003).
For the present case of Dirichlet process priors there
are no such surprises, however, as long as inference
about $F$ is concerned, as one may prove the following.
Here the role of the maximum likelihood estimator is played
by the empirical distribution $F_n$, with Bayes estimator
(posterior mean) equal to $\tilda F_n=(a/(a+n))F_0+(n/(a+n))F_n$.
Below, $W^0$ is a Brownian bridge, i.e.~a Gau{\ss}ian
zero-mean process on $[0,1]$ with covariance structure
$t_1(1-t_2)$ for $t\le t_2$.

\begin{proposition}
\labelloc{p7.1}
Assume the Dirichlet process with
parameter $aF_0$ is used for the distribution of i.i.d.~data
$X_1,X_2,\ldots$, and assume that the real generating mechanism for
these observations is a distribution $F_\true$. Then (i) the process
$\rootn\{F_n(t)-F_\true(t)\}$ converges to $W^0(F_\true(t))$; (ii)
the difference $\rootn(\tilda F_n-F_n)$ goes to zero; and (iii) the
posterior distribution process $V_n(t)=\rootn\{F(t)-\tilda
F_n(t)\}\midd\data$ also converges, with probability one, to
$W^0(F_\true(t))$. The convergence is w.r.t.~the Skorokhod topology
in the space of right-continuous functions with left hand limits.
\end{proposition}

\begin{proof}
The first result is classic and may be found in
e.g.~Billingsley (1968, Ch.~4). The second statement
is immediate from the explicit representation of $\tilda F_n$.
Proving the third involves showing finite-dimensional
convergence in distribution and tightness, as per
the theory of convergence of probability measures
laid out in e.g.~Billingsley (1968).

To show finite-dimensional convergence we start with
$t_1<\cdots<t_m$ and work with differences
$\Delta V_{n,j}=\rootn\{F(t_{j-1},t_j]-\tilda F_n(t_{j-1},t_j]\}$.
The vector of $D_j=F(t_{j-1},\allowbreak t_j]$
has a Dirichlet distribution
with parameters $(n+a)\tilda F_n(t_{j-1},t_j]$.
Also, on a set $\Omega$ of probability one,
both $F_n$ and $\tilda F_n$ tend uniformly to $F_\true$,
by the Glivenko--Cantelli theorem. Finishing this part
of the proof is therefore more or less equivalent to the
following lemma: If $(U_1,\ldots,U_m)$ is a Dirichlet
distributed vector with parameters $(kp_1,\ldots,kp_m)$,
where $p_1+\cdots+p_m=1$, then the vector with components
$(k+1)^{1/2}(U_j-p_j)$ tends with growing $k$ to a multinormal
vector with mean zero and `multinomial' covariance structure
$p_i(\delta_{i,j}-p_j)$, writing $\delta_{i,j}=I_{\{i=j\}}$.
Proving this 
can be done via Scheff\'e's theorem
on convergence of densities,
or more easily via the representation $U_j=G_j/(G_1+\cdots+G_m)$
in terms of independent $G_j\sim{\rm Gamma}(kp_j,1)$,
and for which one quickly establishes that $k^{1/2}(G_j/k-p_j)$
tends to a normal $(0,p_j)$.

It remains to demonstrate the almost sure tightness of $V_n$.
For this purpose, take first $(U,V,W)$ to be Dirichlet with
parameter $(kp,kq,kr)$, where $p+q+r=1$. Then some fairly
long calculations show that
$$\E(U-p)^2(V-q)^2={pq\over (k+1)(k+2)(k+3)}\{k-(k-6)(p+q-3pq)\}. $$
Applying this to the posterior process,
writing $V_n(s,t]=V_n(t)-V_n(s)$ and so on, shows that
$\E \{V_n(s,t]^2V_n(t,u]^2\midd\data\}$
is bounded by $3\tilda F_n(s,t]\tilda F_n(t,u]$, 
with the right hand side converging, under $\Omega$,
towards a quantity bounded by $3\,F_\true(s,u]^2$.
Tightness now follows from the proof of Theorem 15.6
(but not quite by Theorem 15.6 itself) in Billingsley (1968).
\end{proof}

The result above was also in essence proved in Hjort (1991),
and is also related to large-sample studies of
the Bayesian bootstrap, see e.g.~Lo (1987).
We also note that $(n+a+1)^{1/2}$ is a somewhat
superior scaling, compared to $\rootn$,
giving exactly matched first and second moments
for the posterior process.

We further note that the above conclusions hold also
when the strength parameter $a$ of the prior is allowed
to grow with $n$, as long as $a/\rootn\arr0$.
In the more drastic case when $a=cn$, say,
the frequentist and Bayesian schemes do not agree
asymptotically, as $\tilda F_n$ goes a.s.~to
$F_\infty=(c/(c+1))F_0+(1/(c+1))F_\true$. But the
arguments regarding (iii) still go through, showing
that the posterior distribution of $(n+a+1)^{1/2}(F-\tilda F_n)$
tends a.s.~to that of $W^0(F_\infty(\cdot))$.

\subsection{Behaviour of the posterior quantile process}

Here we aim at obtaining results as above for the
quantile processes involved. For the quantiles,
the natural frequentist estimator is $F_n^{-1}$,
while several Bayesian schemes may be considered,
including $\tilda F_n^{-1}$ and the posterior mean
function $\hatt Q_a(y)$ and its natural non-informative
limit $\hatt Q_0(y)$.

\begin{proposition}
\labelloc{p7.2}
Assume, in addition to conditions listed in Proposition
\refloc{p7.1}, that the $F_\true$ distribution has a positive and
continuous density $f_\true$, and let $Q_\true(y)$ and
$q_\true(y)=1/f_\true(Q_\true(y))$ be the true quantile and quantile
density functions. Then (i) the process
$\rootn\{F_n^{-1}(y)-Q_\true(y)\}$ tends to $q_\true(y)W^0(y)$; (ii)
the difference $\rootn\{F_n^{-1}(y)-\tilda F_n^{-1}(y)\}$ goes to
zero in probability; and (iii) the posterior distribution process
$\rootn\{Q(y)-\tilda F_n^{-1}(y)\}\midd\data$ converges a.s.~to the
same limit $q_\true(y)W^0(y)$. The convergence takes place in each
of the spaces $D[\eps,1-\eps]$ of left-continuous functions with
right-hand limits, equipped with the Skorokhod topology, where
$\eps\in(0,\half)$. 
\end{proposition}

\begin{proof}
The first result is again classic, see e.g.~Shorack and Wellner
(1986, Ch.~3). It is typically proven by tending to
the uniform case first, involving say $F_{n,\rm unif}^{-1}(y)$,
and then applying the delta method using the representation
$F_n^{-1}(y)=Q_\true(F_{n,\rm unif}^{-1}(y))$.
Results (ii) and (iii) may be proven in different ways,
but the apparently simplest route is via the method
devised by Doss and Gill (1992), which acts as a
functional delta method operating on the inverse functional
$F\mapsto Q=F^{-1}$. We saw above that
$\rootn\{F(t)-\tilda F_n(t)\}\midd\data$ tends a.s.~to
$V(t)=W^0(F_\true(t))$. From a slight extension of
Doss and Gill's Theorem 2, employing the set $\Omega$
of probability 1 encountered in the previous proposition,
follows that $\rootn\{Q(y)-\hatt F_n^{-1}(y)\}\midd\data$
must tend a.s.~to the process $-V(Q_\true(y))/f_\true(Q_\true(y))$,
which is the same as $-q_\true(y)W^0(y)$. This proves
(iii), since by symmetry $W^0$ and $-W^0$ have identical
distributions. Statement (ii) follows similarly
from Doss and Gill (op.~cit., Theorem~1), again with
the slight extension to secure an `almost sure' version
rather than an `in probability' version,
since the process $\rootn(F_n-\tilda F_n)$
has the zero process as its limit.
\end{proof}

\begin{remark} 
\labelloc{r7.1} 
We also note that $\rootn(\hatt Q_a-\hatt
Q_0)\arr_p0$ follows, by the same type of arguments, starting from
$\rootn(\tilda F_n-F_n)\arr_p0$. In particular, different Bayesians
using different Dirichlet process priors will all agree
asymptotically. Also, the two estimators $\hatt Q_0$ (the
\Bernstein{} smoothed quantiles) and $F_n^{-1}$ (the direct
quantiles) become equivalent for large samples, in the sense of
$\rootn(\hatt Q_0-F_n^{-1})\arr_p0$. This also follows from work of
Sheather and Marron (1990) about kernel smoothing of quantile
functions; see also Cheng (1995).~\square
\end{remark}

An important consequence of the proposition is that the posterior
variance of $\rootn(Q-F_n^{-1})$ tends to the variance of $q_\true
W^0$. This is valid for each Dirichlet strength parameter $a$, as
$n\arr\infty$. For $a=0$, $n$ times the posterior variance $\hatt
V_0(y)$ of (\refloc{eq:V0}) converges a.s.~to $q_\true(y)^2y(1-y)$.
This fact, which may also be proved via results of Conti (2004), is
among the ingredients necessary to secure that the natural
confidence bands $\hatt Q_0\pm z_0\,\hatt V_0^{1/2}$ have the
correct limiting coverage level. This comment also applies to
constructions in the following subsection.

\subsection{Doksum's shift and Parzen's comparison}

Here we first state results for the natural
nonparametric estimators $\tilda D(x)$ and $\tilda\pi(y)$
of Doksum's shift function $D(x)$ and Parzen's comparison
distribution, respectively, before we go on to describe
the behaviour of their Bayesian cousins, introduced in
Section~6. For data $X_1,\ldots,X_n$ from $F_\true$
and $X_1',\ldots,X_m'$ from $G_\true$, let again
$F_n$ and $G_m$ be the empirical distribution functions.
We write $N=n+m$ and assume that $n/N\arr c$ and $m/N\arr 1-c$
as the sample sizes increase. Here $F_\true$ and $G_\true$
are the real underlying distributions, for which
we used Dirichlet process priors $\Dir(aF_0)$ and $\Dir(bG_0)$
in Section~6.

The Doksum estimator is $\tilda D(x)=G_m^{-1}(F_n(x))-x$. Some
analysis, involving the frequentist parts of Propositions
\refloc{p7.1} and \refloc{p7.2}, shows that the $N^{1/2}\{\tilda
D(x)-D_\true(x)\}$ process tends to 
\beq 
\labelloc{eq:asDoksum}
(G_\true^{-1})'(F_\true(x))\hskip-14pt
   &&\{(1-c)^{-1/2}W^0_1(F_\true(x))+c^{-1/2}W^0_2(F_\true(x))\} \nonumber \\
&=&\{c(1-c)\}^{-1/2}(G_\true^{-1})'(F_\true(x))W^0(F_\true(x)),
\eeq
where $D_\true(x)=G_\true^{-1}(F_\true(x))-x$
and $W^0_1$ and $W^0_2$ are two independent Brownian
bridges; these combine as indicated into one such
Brownian bridge $W^0$. This result was given in
Doksum (1974a), and underlies various methods for
obtaining pointwise and simultaneous confidence
bands for $D(x)$; see also Doksum and Sievers (1976).

Arguments used to reach the limit result above may now be repeated
mutatis mutandis, in combination with the \Bernstein--von Mises
results in Propositions \refloc{p7.1}--\refloc{p7.2}, to reach
\begin{equation}
\labelloc{eq:asDoksum2} N^{1/2}\{D(x)-\tilda D(x)\}\midd\data\arr_d
Z_D(x),
\end{equation}
say, using $Z_D$ to denote the limit process in
(\refloc{eq:asDoksum}). The convergence takes place in each
Skorokhod space $D[a,b]$ over which the underlying densities
$f_\true$ and $g_\true$ are positive, and holds with probability 1,
i.e.~for almost all sample sequences. Result (\refloc{eq:asDoksum2})
is valid for the informative case with $a$ and $b$ positive (but
fixed) as well as for the limiting case where
$F\midd\data\sim\Dir(nF_n)$ and $G\midd\data\sim\Dir(mG_m)$. It is
also valid with $\tilda D(x)$ replaced by either the posterior mean
$\hatt D_0(x)$ or posterior median $K_{m,n}^{-1}(\half)$ estimators
discussed in Section~6.

Similarly, the nonparametric Parzen estimator is
$\tilda\pi(y)=G_m(F_n^{-1}(y))$, and a decomposition into two
processes shows with some analysis that 
$N^{1/2}\{\tilda\pi(y)-\pi_\true(y)\}$ tends
to the process 
\beq 
\labelloc{eq:asParzen}
Z_P(y)
&=&{1\over (1-c)^{1/2}}W^0_1(G_\true(F_\true^{-1}(y)))
    +{1\over c^{1/2}}{g_\true(F_\true^{-1}(y))
   \over f_\true(F_\true^{-1}(y))}W^2_0(y) \nonumber \\
\quad &=&(1-c)^{-1/2}W^0_1(\pi_\true(y))+c^{-1/2}\pi_\true'(y)
   W^0_2(y),
\eeq
with $\pi_\true(y)=G_\true(F_\true^{-1}(y))$.
For the case $F_\true=G_\true$, one has $\pi_\true(y)=y$,
and the limit result translates to the quite simple
$(mn/N)^{1/2}(\tilda\pi-\pi)\arr_dW^0$.
This provides an easy and informative way
of checking and testing proximity of two distributions
via the $\tilda\pi$ plot.
``Why aren't people celebrating these facts?'',
as says Parzen in the interview with Newton (2002, p.~373).
Similarly worthy of celebrations, in the Bayesian camp, should be
the fact that (\refloc{eq:asParzen}) has a sister parallel in the
present context, namely that 
$N^{1/2}\{\pi(y)-\hatt\pi(y)\}\midd\data$
tends to the same limit process as in (\refloc{eq:asParzen}).
Here $\hatt\pi(y)$ can be the posterior median estimator or
the posterior mean estimator found in Section~6.

\section{Quantile regression}

Consider the regression situation where certain covariates
$(x_{i,1},\ldots,x_{i,p})^\tr=x_i$ are available for individual $i$,
thought to influence the distribution of $Y_i$. Assume that
$Y_i=\beta^\tr x_i+\sigma\eps_i$, where
$\beta=(\beta_1,\ldots,\beta_p)^\tr$ contains unknown
regression parameters and $\eps_1,\ldots,\eps_n$
are independent error terms, coming from a scaled
residual distribution $F$. Thus a prospective observation
$Y$, with covariate information $x$,
will have distribution $F(t\midd x)=F((t-\beta^\tr x)/\sigma)$,
conditional on $(\beta,\sigma,F)$. Its quantile function becomes
$Q(u\midd x)=\beta^\tr x+\sigma\,Q(u)$,
writing again $Q$ for $F^{-1}$.

The problem to be discussed now is that of Bayesian inference for
$Q(u\midd x)$, starting out with a prior for $(\beta,\sigma,F)$.
Take $(\beta,\sigma)$ and $F$ to be independent, with
a prior density $\pi(\beta,\sigma)$ and a $\Dir(aF_0)$
prior for $F$, where the prior guess $F_0$ has a density $f_0$.
The posterior distribution of $(\beta,\sigma,F)$ may then
be described as follows.
First, the posterior density of $\beta$ can be shown to be
$$\pi(\beta,\sigma\midd\data)={\rm const.}\,\pi(\beta,\sigma)
  \prod_{\rm distinct} f_0((y_i-\beta^\tr x_i)/\sigma), $$
where the product is taken over distinct values of $y_i-\beta^\tr x_i$.
This may be shown via techniques in Hjort (1986).
Secondly, given data and $(\beta,\sigma)$,
$Q$ acts as the posterior quantile
process from a Dirichlet $F$ with parameter
$aF_0+\sumin\delta((y_i-\beta^\tr x_i)/\sigma)$,
with $\delta(z)$ denoting unit point mass at $z$;
in particular, expressions for
$\hatt Q_a(u\midd\beta,\sigma)=\E\{Q(u)\midd\beta,\sigma,\data\}$
may be written down using the results of earlier sections.

In combination, this gives for each $x_0$
an estimator for $Q(u\midd x_0)$ of the form
\beqn
\hatt Q_a(u\midd x_0)
&=&\E\{\beta^\tr x_0+\sigma Q(u)\midd\data\} \\
&=&\hatt\beta^\tr x_0+\E\{\sigma\hatt Q_a(u\midd\beta,\sigma)\midd\data\} \\
&=&\hatt\beta^\tr x_0+\int \sigma\hatt Q_a(y\midd\beta,\sigma)
  \pi(\beta,\sigma\midd\data)\,\d\beta\,\d\sigma,
\eeqn
where $\hatt\beta$ is the posterior mean of $\beta$.
For the particular case of $a$ tending to zero, this gives
\beqn
\hatt Q_0(u\midd x_0)
=\hatt\beta^\tr x_0
   +\sumin {n-1\choose i-1}u^{i-1}(1-u)^{n-i}\,e_i.
\eeqn
Here $e_i=\int (y-\beta^\tr x)_{(i)}\pi(\beta\midd\data)\,\d\beta$,
where, for each $\beta$, $(y-\beta^\tr x)_{(i)}$ is the
result of sorting the $n$ values of $y_j-\beta^\tr x_j$ and
then finding the $i$th ranked one. 
The simplest implementation might be to draw a large number of
$\beta$s from the posterior density,
and then for each of these sort the values of
$y_j-\beta^\tr x_j$. Averaging over all simulations
then gives $e_i$ as the posterior mean of $(y-\beta^\tr x)_{(i)}$,
for each $i=1,\ldots,n$, and in their turn
$\hatt Q_0(u\midd x_0)$ for all $x_0$.

One may also give a separate recipe for making
inference for $Q$, the residual quantile process.
Other Bayesian approaches to quantile regression
are considered in Kottas and Gelfand (2001)
and Hjort and Walker (2006).

\section{Concluding remarks}

In our final section we offer some concluding
comments, some of which might point to further
problems of interest.

\medskip
{\it Other priors.}
There are of course other possibilities for quantifying
prior opinions of quantile functions. One may e.g.~start
with a prior more general than or different from the
Dirichlet process for $F$, like Doksum's (1974b)
neutral to the right processes, 
or mixtures of Dirichlet processes, 
and attempt to reach results
for the consequent quantile processes $Q=F^{-1}$.
Another and more direct approach is
via the versatile class of quantile pyramid processes
developed in Hjort and Walker (2006). These work by first
drawing the median $Q(\half)$ from a certain distribution;
then the two other quartiles $Q({1\over 4})$ and $Q({3\over 4})$
given the median; then the three remaining octiles
$Q({j\over 8})$ for $j=1,3,5,7$; and so on.
The Dirichlet process can actually be seen to be a special
case of these pyramid constructions.
While the treatment in Hjort and Walker leads to recipes
which can handle the prior to posterior updating task
for any quantile pyramid, this relies on simulation techniques
of the McMC variety. Part of the contribution of the
present chapter is that explicit formulae and characterisations
are developed, partly obviating the need for such simulation
work, for the particular case of the Dirichlet processes.

\medskip
{\it An invariance property.}
Our canonical Bayes estimator (\refloc{eq:Q0hat}) was derived by
starting with a $\Dir(aF_0)$ prior for $F$ and then letting $a$ go
to zero. Extending the horizon beyond the simple i.i.d.~setting,
suppose for illustration that data are assumed to be of the form
$X_i=\xi+\sigma Z_i$, with $Z_i$ having distribution $G$. One may
then give a semiparametric prior for the distribution
$F(t)=G((t-\xi)/\sigma)$ of $X_i$, with a prior for $(\xi,\sigma)$
and an independent $\Dir(aG_0)$ prior for $G$. This leads to a more
complicated posterior distribution for $Q(y)=\xi+\sigma Q_G(y)$,
say. But since $G$ given data and the parameters is a Dirichlet with
parameter $aG_0+\sumin\delta((x_i-\mu)/\sigma)$, results of
Sections~2 and 3 give formulae for
$\E\{Q(y)\midd\data,\xi,\sigma\}$. For the non-informative case of
$a=0$,
$$\E\{Q(y)\midd\data,\xi,\sigma\}
  =\xi+\sigma\sumin{n-1\choose i-1}y^{i-1}(1-y)^{n-1}
   {x_{(i)}-\xi\over \sigma}. $$
But the extra parameters cancel out, showing that the posterior mean
is again the (\refloc{eq:Q0hat}) estimator, which therefore is the
limiting Bayes rule for rather wider classes of priors than only the
pure Dirichlet. The argument goes through for each monotone
transformation $X_i=a_\theta(Z_i)$ with a prior for $(\theta,G)$.

In situations where the Lorenz curve and Gini index are of
interest, for example, one might think of data as
$X_i=\theta Z_i$, with separate priors for $\theta$
and the distribution $G$ of $Z_i$. The above argument
shows that the $\theta$ information is not relevant
for $Q(y)=\theta Q_G(y)$, when $a$ is small, thus lending further
support to the estimators $\hatt L_0$ and $\hatt G_0$
of Section~5.

\medskip
{\it Alternative proofs.}
There are other venues of interest towards proving Proposition
\refloc{p7.2} or other versions thereof. Johnson and Sim (2006) give
a different proof of the large-sample joint normality of a finite
number of posterior quantiles, including asymptotic expansions.
Conti (2004) has independently of the present authors reached
results for the posterior process $\rootn(Q-\tilda F_n^{-1})$,
partly using strong Hung{\H a}rian representations. His approach
gives results that are more informative than Proposition
\refloc{p7.2} concerning the boundaries, i.e.~for $y$ close to 0 and
$y$ close to 1, where our direct method works best on
$D[\eps,1-\eps]$ for a fixed small $\eps$.
Another angle is to exploit approximations to 
the Beta and Dirichlet distributions
associated with the random $F$ and turn these around 
to good approximations for $Q$. 
%
A third possibility of interest is to express the random posterior
quantile process as $Q(y)=x_{(N(y))}$, with $N(y)$ the random
process described in Section~2.4, climbing from $1$ at zero to $n$
at one. One may show that $\rootn\{N(y)/n-y\}$ tends to a Brownian
bridge, and couple this with $Q(y)=Q_n(N(y)/n)$ to give yet another
proof of the \Bernstein--von Mises part of Proposition
\refloc{p7.2}.

\medskip
{\it Simultaneous confidence bands.}
In our illustrations we focussed on confidence bands with correct
pointwise coverage. One may also construct simultaneous bands for
the different situations, with some more work. For the Doksum shift
function, in the frequentist setting, such simultaneous bands were
constructed in Doksum (1974a), Doksum and Sieverts (1976) and
Switzer (1976). To match this in the Bayesian setting, one might
simulate a large number of $D(x)$ curves from the posterior process,
and note the quantiles of the distribution of simulated
$\max_{[a,b]}|D(x)-\hatt D_0(x)|$ across some interval $[a,b]$ of
interest. Another method, using result (\refloc{eq:asDoksum2}), 
is to note that 
$N^{1/2}\max_{a\le x\le b}|D(x)-\hatt D_0(x)|\,\midd\data$
tends in distribution to 
\beqn
\max_{a\le x\le b}|Z_D(x)|
={1\over \{c(1-c)\}^{1/2}}
   \max_{F(a)\le v\le F(b)}{|W^0(v)|\over g_\true(G_\true^{-1}(v))}.
\eeqn
With appropriate consistent estimation of the denumerator
one might simulate the required quantile of the limiting
distribution. Other bands evolve with alternative
weight functions.

\medskip
{\it Further quantilian quantities.}
There are yet other statistical functions or parameters
of interest that depend on quantile functions
and that can be worked with using methods from
our chapter. One such quantity is the total time on test statistic
$T(u)=\int_0^{Q(u)}\{1-F(x)\}\,\d x$.
Doksum and James (2004) show how inference
for $T$ may be carried out via Bayesian bootstraps.

\medskip
{\it More informative priors for two-sample problems.}
In situations where the Doksum band contains a horizontal line it
indicates that the shift function is nearly constant, which
corresponds to a location translation from $F$ to $G$, say
$G(t)=F(t-\delta)$. For the Doksum--Bjerkedal data analysed in
Figure~3 
the band nearly contains a linear curve, which
indicates a location-and-scale translation, say
$G(t)=F((t-\delta)/\tau)$. The present point is that it is fruitful
to build Bayesian prior models for such scenarios, linking $F$ and
$G$ together, as opposed to simply assuming prior independence of
$F$ and $G$. One version is to take $F\sim\Dir(aF_0)$ and then
$G(t)=F((t-\delta)/\tau)$ with a prior for $(\delta,\tau)$. This
leads to fruitful posterior models for $(F,\delta,\tau)$.


\section*{Appendix: various proofs}

{\it Relation between Beta cumulatives.}
Let $\be(\cdot;a,b)$ and $\Be(\cdot;a,b)$ denote the density and cumulative
distribution of a Beta variable with parameters $(a,b)$.
Then, by partial integration, for $b>1$, 
$$\Be(c;a,b)-\Be(c;a+1,b-1)={\be(c;a+1,b)\over a+b}
            ={\be(1-c;b,a+1)\over a+b}. \eqno({\rm A1})$$

\medskip
{\it Proof of Proposition \refloc{p2.1}.}
There are several ways in which to prove this, including analysis
via Taylor type expansions of the (\refloc{eq:deltaHna})
probabilities and their sum; see also Conti (2004). Here we briefly
outline another and more probabilistic argument. The idea is to
decompose the posterior distribution of $F$ in two parts,
corresponding to jumps $D_1,\ldots,D_n$ at the data points and a
total probability $E=F(\RR-\{x_1,\ldots,x_n\})$ representing all
increments between the data points. Thus
$$F(t)=\sumin D_iI\{x_{(i)}\le t\}+\sumin E_iI\{x_{(i)}\le t\}
   =\tilda F(t)+F^*(t), $$
say, with $E_i$ the part of $E$ corresponding to
the window $(x_{(i-1)}, x_{(i)})$ between data points.
The point here is that $(D_1,\ldots,D_n,E)$
has a Dirichlet $(1,\ldots,1,a)$ distribution,
with $E$ becoming small in size as $n$ increases.
In fact, $E\le a/\rootn$ with probability at least $1-1/\rootn$.
Thus $F=\tilda F+F^*$ with $F-\tilda F\le a/\rootn$,
with high probability, and $Q=F^{-1}$ must with a high probability
be close to $\tilda Q=\tilda F^{-1}$. But the latter has all
its jumps exactly situated at the data points. \square

\medskip
{\it Proof of Proposition \refloc{p3.1}.}
We first recall that for any cumulative distribution function $H$ on
the real line,
$$\int_0^\infty x\,\d H(x)=\int_0^\infty\{1-H(x)\}\,\d x,
  \quad
  \int_{-\infty}^0 x\,\d H(x)=-\int_{-\infty}^0 H(x)\,\d x. $$
These results can be shown using partial integration and the Fubini
theorem, and hold in the sense that finiteness of one integral
implies finiteness of the sister integral, and vice versa. These
formulae are what is being used when we in Section~3 preferred
formula (\refloc{eq:Qhatformula}) to (\refloc{eq:Qahat}).

With the above formulae and characterisations we learn
that the finite existence of the posterior mean of $Q(y)$
hinges on the finiteness of the extreme parts
$\int_c^\infty\Be(y;aF_0(x)+n,a\bar F_0(x))\,\d x$,
for $c\ge x_{(n)}$,
and $\int_{-\infty}^b\Be(1-y;a\bar F_0(x)+n,aF_0(x))\,\d x$,
for $b\le x_{(1)}$.
Using $\Gamma(v)=\Gamma(v+1)/v$
the first integral may be expressed as
$$\int_c^\infty{\Gamma(a+n)a\bar F_0(x)
   \over \Gamma(aF_0(x)+n)\Gamma(a\bar F_0(x)+1)}
   \Bigl[\int_0^y u^{aF_0(x)+n-1}(1-u)^{a\bar F_0(x)-1}\,\d u\Bigr]\,\d x, $$
which is of the form $\int_c^\infty a\bar F_0(x)g(x)\,\d x$
for a bounded function $g$; hence this the integral is finite
if and only if $\int_c^\infty\{1-F_0(x)\}\,\d x$ is finite.
We may similarly show that the second integral is finite
if and only if $\int_{-\infty}^b F_0(x)\,\d x$ is finite.
These arguments are valid for any $n$, also for the
no-sample prior case of $n=0$. This proves the proposition.
\square

\section*{Acknowledgements}

The authors gratefully acknowledge support and hospitality from the
Department of Mathematics at the University of Oslo and the Istituto
di Metodi Quantitativi at Bocconi University in Milano, at
reciprocal research visits. 
Constructive comments from 
Dorota Dabrowska, Alan Gelfand, Pietro Muliere, 
Vijay Nair and Stephen Walker 
have also been appreciated. 

\def\annstat{Annals of Statistics}
\def\annprob{Annals of Probability}
\def\jasa{Journal of the American Statistical Association}
\def\statscience{Statistical Science}


\begin{thebibliography}{000}

\bibitem {}
\textsc{Billingsley, P.} (1968).
{\sl Convergence of Probability Measures.}
Wiley, New York.

\bibitem {}
\textsc{Bickel, P.J.~and Doksum, K.A.} (2001).
{\sl Mathematical Statistics: Basic Ideas and Selected Topics}
(2nd ed.), Volume 1.
Prentice Hall, Upper Saddle River, New Jersey.

\bibitem {}
\textsc{Bjerkedal, T.} (1960).
Acquisition of resistance in guinea pigs infected
with different doses of virulent tubercle bacilli.
{\sl American Journal of Hygiene} {\bf 72}, 132--148.


\bibitem {}
\textsc{Cheng, C.} (1995).
The Bernstein polynomial estimator of a smooth quantile function.
{\sl Statistics and Probability Letters} {\bf 24}, 321--330.

\bibitem {}
\textsc{Conti, P.L.} (2004).
Approximated inference for the quantile function
via Dirichlet processes.
{\sl Metron} {\bf LXII}, 201--222.

\bibitem {}
\textsc{Diaconis, P.~and Freedman, D.A.} (1986a).
On the consistency of Bayes estimates [with discussion].
{\sl\annstat} {\bf 14}, 1--67.

\bibitem {}
\textsc{Diaconis, P.~and Freedman, D.A.} (1986b).
On inconsistent Bayes estimates of location.
{\sl\annstat} {\bf 14}, 68--87.

\bibitem {}
\textsc{Doksum, K.A.} (1974a).
Empirical probability plots and statistical inference
for nonlinear models in the two-sample case.
{\sl\annstat} {\bf 2}, 267--277.

\bibitem {}
\textsc{Doksum, K.A.} (1974b).
Tailfree and neutral random probabilities
and their posterior distributions.
{\sl\annprob} {\bf 2}, 183--201.

\bibitem {}
\textsc{Doksum, K.A.~and Sievers, G.L.} (1976).
Plotting with confidence: Graphical comparisons
of two populations.
{\sl Biometrika} {\bf 63}, 421--434.

\bibitem {}
\textsc{Doksum, K.A.~and James, L.F.} (2004).
On spatial neutral to the right processes
and their posterior distributions.
In {\sl Mathematical Reliability:
An Expository Perspective} (eds.~R.~Soyer,
T.A.~Mazzuchi and N.D.~ Singpurvalla),
Kluwer International Series, 87--104.

\bibitem {}
\textsc{Doss, H.~and Gill, R.D.} (1992).
An elementary approach to weak convergence for quantile processes,
with applications to censored survival data.
{\sl\jasa} {\bf 87}, 869--877.

\bibitem {}
\textsc{Ferguson, T.S.} (1973).
A Bayesian analysis of some nonparametric problems.
{\sl\annstat} {\bf 1}, 209--230.

\bibitem {}
\textsc{Ferguson, T.S.} (1974).
Prior distributions on spaces of probability measures.
{\sl\annstat} {\bf 2}, 615--629.

\bibitem {}
\textsc{Hjort, N.L.} (1986).
Discussion contribution to P.~Diaconis and D.~Freedman's paper
`On the consistency of Bayes estimates',
{\sl\annstat} {\bf 14}, 49--55.

\bibitem {}
\textsc{Hjort, N.L.} (1991).
Bayesian and empirical Bayesian bootstrapping.
Statistical Research Report, University of Oslo.

\bibitem {}
\textsc{Hjort, N.L.} (1996).
Bayesian approaches to non- and semiparametric density estimation
[with discussion].
In {\sl Bayesian Statistics 5},
proceedings of the Fifth International Val\`encia Meeting
on Bayesian Statistics (eds.~J.~Berger,
J.~Bernardo, A.P.~Dawid, A.F.M.\allowbreak~Smith),
223--253. Oxford University Press.

\bibitem {}
\textsc{Hjort, N.L.} (2003).
Topics in nonparametric Bayesian statistics [with discussion].
In {\sl Highly Structured Stochastic Systems}
(eds.~P.J.~Green, S.~Richardson and N.L.~Hjort),
Oxford University Press.

\bibitem {}
\textsc{Hjort, N.L.~and Walker, S.G.} (2006).
Quantile pyramids for Bayesian nonparametrics.
{\sl\annstat}, to appear.

\bibitem {}
\textsc{Johnson, R.A.~and Sim, S.} (2006).
Nonparametric Bayesian inference about percentiles.
This volume.

\bibitem {}
\textsc{Kottas, A.~and Gelfand, A.} (2001).
Bayesian semiparametric median regression modeling.
{\sl \jasa} {\bf 96}, 1458--1468.

\bibitem {}
\textsc{LeCam, L.~and Yang, G.L.} (1990).
{\sl Asymptotics in Statistics.}
Springer-Verlag, New York.


\bibitem {}
\textsc{Lo, A.Y.} (1987).
A large-sample study of the Bayesian bootstrap.
{\sl\annstat} {\bf 15}, 360--375.

\bibitem {}
\textsc{Lorenz, M.C.} (1905).
Methods of measuring the concentration of wealth.
{\sl\jasa} {\bf 9}, 209--219.

\bibitem {}
\textsc{Laake, P., Laake, K.~and Aaberge, R.} (1985).
On the problem of measuring the distance between distribution functions:
Analysis of hospitalization versus mortality.
{\sl Biometrics} {\bf 41}, 515--523.

\bibitem {}
\textsc{Newton, H.J.} (2002).
A conversation with Emanuel Parzen.
{\sl\statscience} {\bf 17}, 357--378.
Correction, op.~cit., 467.

\bibitem {}
\textsc{Parzen, E.} (1979).
Nonparametric statistical data modeling [with discussion].
{\sl\jasa} {\bf 74}, 105--131.

\bibitem {}
\textsc{Parzen, E.} (1982).
Data modeling using quantile and density-quantile functions.
{\sl Some recent advances in statistics},
Symposium Lisbon 1980, 23--52.

\bibitem {}
\textsc{Parzen, E.} (2002).
Discussion of Breiman's `Statistical modeling: The two cultures'.
{\sl\statscience} {\bf 16}, 224--226.

\bibitem {}
\textsc{Sheather, S.J.~and Marron, J.S.} (1990).
Kernel quantile estimation.
{\sl\jasa} {\bf 80}, 410--416.

\bibitem {}
\textsc{Shorack, G.R.~and Wellner, J.} (1986).
{\sl Empirical Processes With Applications to Statistics.}
Wiley, New York.

\bibitem {}
\textsc{Switzer, P.} (1976).
Confidence procedures for two samples.
{\sl Biometrika} {\bf 53}, 13--25.

\bibitem {}
\textsc{Aaberge, R.} (2001).
Axiomatic characterization of the Gini coefficient and Lorenz curve
orderings.
{\sl Journal of Economic Theory} {\bf 101}, 115--132.
Correction, ibid.

\bibitem {}
\textsc{Aaberge, R., Bjerve, S.~and Doksum, K.A.} (2005).
Lorenz, Gini, Bonferroni and quantile regression.
Unpublished manuscript.

\end{thebibliography}
\end{document}